\pdfoutput=1
\documentclass[11pt]{amsart}
\usepackage[paper,pkg/amsaddr=true]{zbs}

\newcommand{\cA}{\mathcal{A}}
\newcommand{\cZ}{\mathcal{Z}}
\newcommand{\cP}{\mathcal{P}}
\newcommand{\tors}{\mathrm{tors}}
\DeclareMathOperator{\conv}{conv}
\DeclareMathOperator{\spanop}{span}

\title{Hypergraphic zonotopes and acyclohedra}
\author[Cosmin Pohoata]{Cosmin Pohoata\nfts{1}}
\address{\nfts{1}Department of Mathematics, Emory University, Atlanta, GA 30322, USA}
\email{cosmin.pohoata@emory.edu}
\author[Daniel G.\ Zhu]{Daniel G.\ Zhu\nfts{2}}
\address{\nfts{2}Department of Mathematics, Princeton University, Princeton, NJ 08544, USA}
\email{zhd@princeton.edu}
\begin{document}
\begin{abstract}
    We introduce a higher-uniformity analogue of graphic zonotopes and permutohedra. Specifically, given a $(d+1)$-uniform hypergraph $H$, we define its \vocab{hypergraphic zonotope} $\cZ_H$, and when $H$ is the complete $(d+1)$-uniform hypergraph $K^{(d+1)}_n$, we call its hypergraphic zonotope the \vocab{acyclohedron} $\cA_{n,d}$. 
    
    We express the volume of $\cZ_H$ as a homologically weighted count of the spanning $d$-dimensional hypertrees of $H$, which is closely related to Kalai's generalization of Cayley's theorem in the case when $H=K^{(d+1)}_n$ (but which, curiously, is not the same). We also relate the vertices of hypergraphic zonotopes to a notion of acyclic orientations previously studied by Linial and Morganstern for complete hypergraphs.
\end{abstract}

\maketitle

\section{Introduction}
One interesting direction within combinatorics is the study of higher-uniformity analogues of standard objects and theorems in graph theory. Specifically, a graph can be viewed as a $1$-dimensional simplicial complex, and many properties of graphs, most notably connectivity and acyclicity, can be viewed purely topologically. In one of the earliest works in this area, Kalai \cite{Kalai83} generalized these ideas to $d$-dimensional complexes, defining a {\textit{$d$-dimensional hypertree}} to be a $d$-dimensional simplicial complex with a full $(d-1)$-skeleton such that both $\tilde H_{d-1}(T;\setr)=0$ (generalizing connectivity) and $\tilde H_{d}(T;\setr)=0$ (generalizing acyclicity). It is then natural to ask whether one can count the number of such $d$-dimensional hypertrees on $n$ labeled vertices, or equivalently, the number of spanning $d$-dimensional hypertrees of the complete $(d+1)$-uniform hypergraph of order $n$. In \cite{Kalai83}, Kalai found a beautiful formula for this count, but with a somewhat mysterious homological weighting:
\begin{thm}[Kalai \cite{Kalai83}] \label{thm:kalai}
If $n$ and $d$ are positive integers and $\mathcal{T}(n,d)$ is the set of $d$-dimensional hypertrees on $n$ labeled vertices, then
\[\sum_{T\in\mathcal{T}(n,d)} \abs{\tilde H_{d-1}(T; \setz)}^2 = n^{\binom{n-2}{d}}.\]
\end{thm}
Since $\tilde H_0(T; \setz) = 0$ for any $1$-dimensional hypertree, \cref{thm:kalai} generalizes Cayley's formula for the number of trees on $n$ labeled vertices.

A second interesting direction within combinatorics is the study of polytopes that encode combinatorial structures, which can often provide a novel perspective on various phenomena. The quintessential example is the $n$-dimensional hypercube, whose geometry encodes many of the combinatorial properties of the subsets of $\left\{1, \dots, n\right\}$. Another fundamental example is the matching polytope of a given graph, whose geometric features can explain all the known properties of the graph's matchings (for example, whether a graph has a perfect matching or not). One other classical example is the \emph{graphic zonotope} $\cZ_G$ of a graph $G=(V,E)$ with $V=\left\{v_1,\ldots,v_n\right\}$, which is the polytope in $\mathbb{R}^{n}$ defined to be the Minkowski sum of segments which correspond to edges of the graph $G$. Formally, 
\[
    \cZ_G = \sum_{\set{v_i,v_j} \in E} \conv(\left\{e_{i},e_{j}\right\}).
\]
Here and henceforth, $\conv(S)$ will denote the convex hull of a finite set of points $S$. Remarkably, the volume and lattice points
of $\cZ_G$ encode the number of spanning trees and forests in $G$.

\begin{thm}[{see e.g.\ Stanley \cite[Ex.~4.32]{Stanley_1997} and Postnikov \cite[Prop.~2.4]{Postnikov}}]\label{zonoo}
    For a connected graph $G$ on $n$ vertices, the $(n-1)$-dimensional volume of the graphical zonotope $\cZ_G$ equals the number of spanning trees of $G$. The number of lattice points of $\cZ_G$ equals the number of forests in the graph $G$.
\end{thm}

Moreover, the graphic zonotope of the complete graph $K_n$ is the \emph{$n$-permutohedron}
\[\Pi_{n} = \operatorname{conv}\paren[\big]{\setmid[\big]{(\sigma(0),\ldots,\sigma(n-1)) \in \mathbb{R}^{n}}{\sigma \in S_{n}}},\]
a polytope in $\mathbb{R}^{n}$ whose vertices correspond to permutations $S_n$ of the set $\left\{0,1\ldots,n-1\right\}$. (For a proof, see e.g.\ Postnikov \cite[Proposition 2.3]{Postnikov}, where an elegant proof using Newton polytopes is given.) The permutohedron $\Pi_{n}$ lives inside the hyperplane $\setmid{x \in \mathbb{R}^{n}}{x_1+\ldots+x_n = \binom{n}{2}}$, and \cref{zonoo} says that its $(n-1)$-dimensional volume (relative to this hyperplane) equals precisely $n^{n-2}$, the number of spanning trees of the complete graph $K_{n}$.  

This short paper combines these two directions, by defining higher uniformity analogues of graphic zonotopes and permutohedra. Specifically, given a $(d+1)$-uniform hypergraph $H$, we define its \vocab{hypergraphic zonotope} $\cZ_H$, and when $H$ is the complete $(d+1)$-uniform hypergraph $K^{(d+1)}_n$, we call its hypergraphic zonotope an \vocab{acyclohedron} $\cA_{n,d}$. Our main results are generalizations of the two properties of graphic zonotopes mentioned above. In particular, we will first express the volume of $\cZ_H$ as a weighted sum over the spanning hypertrees of $H$, which in the case of the acyclohedron $\cA_{n,d}$ is a weighted sum over $d$-dimensional hypertrees on $n$ labeled vertices. Curiously, these weights are closely related but not identical to those in \cref{thm:kalai}; instead of $\abs{\tilde H_{d-1}(T)}^2$, we instead sum $\abs{\tilde H_{d-1}(T)}$. 

\begin{thm} \label{thm:main}
The volume of $\cZ_H$ is the sum of $\abs{\tilde H_{d-1}(T; \setz)}$, over all spanning hypertrees $T$ of $H$.
\end{thm}

Our second result relates the vertices of hypergraphic zonotopes to a notion of acyclic orientations previously studied by Linial and Morganstern \cite{LinMor} for complete hypergraphs. Thus, these results demonstrate a previously unknown relation between the hypertrees of Kalai and the hypertournaments of Linial and Morganstern. This relation suggests that even in uniformity two it is sometimes more natural to regard the vertices of $\Pi_{n}$ as acyclic orientations of the complete graph on $n$ vertices, rather than permutations of $\left\{0,1,\ldots,n-1\right\}$.

Despite a steady stream of recent works, many basic questions in high-dimensional/high-uniformity combinatorics remain unanswered, and it is our hope that this polytopal perspective could lead to a broader understanding of the area.

\section{Definitions}
A \vocab{simplicial complex} $X$ consists of a set of \vocab{vertices} $V(X)$ (which we will always take to be finite) and a set $S(X)$ of \vocab{simplices}, nonempty subsets of $V(X)$, such that
\begin{itemize}
    \item if $\sigma \in S(X)$ and $\tau$ is a nonempty subset of $\sigma$, then $\tau \in S(X)$;
    \item if $v\in V(X)$, then $\set{v} \in S(X)$.
\end{itemize}
Let $X_k$ denote the simplices in $X$ with $k+1$ vertices. For $k \geq 0$ and a commutative ring $R$ (which in this paper will always be $\setz$ or $\setr$) define the space of \emph{$k$-chains} $C_k(X; R)$ to be the $R$-module freely generated by symbols of the form $(v_0v_1\cdots v_k)$ for $\set{v_0, v_1,\ldots,v_k} \in X_{k+1}$, subject to the relation that for any permutation $\pi$ of $\set{0,1,\ldots,k}$, we have $(v_{\pi(0)}v_{\pi(1)}\cdots v_{\pi(k)}) = \sgn\pi \cdot  (v_0v_1\cdots v_k)$. Let $C^k(X;R)$, the space of \emph{$k$-cochains}, be the dual of $C_k(X;R)$ as $R$-modules. We will treat $C_k(X;\setz)$ and $C^k(X;\setz)$ as subspaces of $C_k(X;\setr)$ and $C^k(X;\setr)$, respectively.

We now define reduced homology.
If $k \geq 1$, define the linear map $\partial_k\colon C_k(X;R) \to C_{k-1}(X;R)$ given by
\[\partial_k((v_0v_1\cdots v_k)) = \sum_{i=0}^k (-1)^i (v_0\cdots \hat v_i \cdots v_k),\]
where hats denote omission. Furthermore, let $\partial_0 \colon C_0(X; R) \to R$ be such that $\partial_0((v)) = 1$ for $v \in X_0$. It is well-known (and easily shown) that $\partial_k$ is well-defined and that $\partial_{k-1} \circ \partial_k = 0$. Thus, the space of \vocab{boundaries} $B_k(X; R) = \im \partial_{k+1}$ is contained in the space of \vocab{(reduced) cycles} $Z_k(X; R) = \ker \partial_k$, so we may define the \vocab{reduced homology group} $\tilde H_k(X; R) = Z_k(X; R)/B_k(X; R)$. For $k \geq 1$, the map $\partial_k$ also induces a dual map $d_{k-1} \colon C^{k-1}(X; R) \to C^k(X; R)$, and analogously we define the space of \vocab{coboundaries} $B^k(X; R) = \im d_{k-1}$.
\begin{rmk}
The standard homology group $H_k(X; R)$ is defined in the same way except that $\partial_0$ is the zero map. This has the consequence that $\tilde H_k(X; R) \cong H_k(X; R)$ for $k \geq 1$ and $\tilde H_0(X; R) \oplus R \cong H_0(X; R)$. However, in this paper we will only ever use reduced homology.
\end{rmk}

Given a $(d+1)$-uniform hypergraph $H$, we may consider it as a simplicial complex with vertex set $V(H)$ and simplices given by all edges of $H$, along with every nonempty subset of $V(H)$ with size at most $d$. We now make the central definition of this paper.
\begin{defn}
The \vocab{hypergraphic zonotope} $\cZ_H\subseteq Z_{d-1}(X; \setr)$ of a $(d+1)$-uniform hypergraph $H$ is the Minkowski sum of the line segments $\conv(\set{0, \partial_d ((v_0v_1\cdots v_d))})$ over all edges $\set{v_0,v_1,\ldots,v_d} \in E(H)$.\footnote{Each edge is only considered once in the sum, with an arbitrary ordering chosen for its vertices. The resulting polytope is only well-defined up to translation by $Z_{d-1}(X;\setz)$.} In the special case when $H$ is the complete $(d+1)$-uniform hypergraph on $n$ vertices, we call $\cZ_H$ an \vocab{acyclohedron} and denote it as $\cA_{n,d}$.    
\end{defn}
Note that when $d = 1$, these notions reduce to the notions of a graphical zonotope and a permutohedron. 

\begin{examp}
The acyclohedron $\cA_{4,2}$ is the polytope given by
\[\conv(\set{0, u_1}) + \conv(\set{0, u_2}) + \conv(\set{0, u_3}) + \conv(\set{0, u_4})\]
where
\begin{align*}
u_1 &=  (v_1v_2)-(v_1v_3)+(v_2v_3) &
u_3 &=  (v_1v_3)-(v_1v_4)+(v_3v_4) \\
u_2 &=  (v_1v_2)-(v_1v_4)+(v_2v_4) &
u_4 &=  (v_2v_3)-(v_2v_4)+(v_3v_4).
\end{align*}
It can be shown that $u_1,u_2,u_3,u_4$ span the $3$-dimensional space $Z_1(K^{(3)}_4;\setr)$, with their only linear relation being $u_1 - u_2 + u_3 - u_4 = 0$. It follows that $\cA_{4,2}$ is a rhombic dodecahedron.
\end{examp}

\section{Hyperforests and the Ehrhart Polynomial}
Given an integral $d$-dimensional polytope $\cP$ in $\mathbb{R}^{n}$ (a polytope whose vertices are points in the $\mathbb{Z}^{n}$ lattice), one important construction is the \vocab{Ehrhart polynomial} $L(\cP, t)$, which, for every integer $t \geq 1$, counts the number of lattice points in the dilate $t\cP$. In \cite{Ehr62}, Ehrhart showed that this quantity is a rational polynomial of degree $d$ in $t$, i.e. there exist rational numbers $L_{0}(\cP),\ldots,L_{d}(\cP)$ such that
\[L(\cP, t)= L_{d}(\cP) t^{d} + L_{d-1}(\cP) t^{d-1} + \ldots + L_{0}(\cP)\]
for all positive integers $t$. 

In this section, we will express the Ehrhart polynomial of $\cZ_H$ in terms of hypertrees and hyperforests. To set things up, observe that $\cZ_H$ lies in $Z_{d-1}(H; \setr)$ and has vertices in $Z_{d-1}(H; \setz)$. As $Z_{d-1}(H; \setz)$ is a lattice within $Z_{d-1}(H; \setr)$, we will treat it as the set of integral points. In particular, we normalize the volume on $Z_{d-1}(H; \setr)$ such that the volume of $Z_{d-1}(H; \setr)/Z_{d-1}(H; \setz)$ is $1$. Furthermore, we define a \vocab{hyperforest} to be a $(d+1)$-uniform hypergraph $H$ with $\tilde H_d(H;\setr) = 0$; recall from the introduction that a hypertree is a hyperforest such that $\tilde H_{d-1}(H;\setr) = 0$. Given a $(d+1)$-uniform hypergraph $H$ define a \vocab{spanning hyperforest} and a \vocab{spanning hypertree} to be a hyperforest and hypertree, respectively, that is contained within $H$ and also has the same vertex set.

\begin{thm} \label{thm:zehr}
We have
\[L(\cZ_H, t) = \sum_{F \subseteq H} \abs{\tilde H_{d-1}(F; \setz)_\tors} \,t^{\abs{E(F)}}.\]
where the sum is over spanning hyperforests $F$ of $H$ and $A_\tors$ denotes the torsion elements of an abelian group $A$.
\end{thm}
\begin{proof}
We recall a classical result on the Ehrhart polynomials of zonotopes.
\begin{thm}[{\cite[Thm.~2.2]{Standegseq}}] \label{thm:ehr}
If $\cZ$ is the Minkowski sum of the line segments $\conv(\set{0, \beta_i})$ for $\beta_1, \beta_2, \ldots, \beta_r \in \setz^n$, then
\[L(\cZ, t) = \sum_X h(X) t^{\abs{X}},\]
where the sum is over linearly independent subsets of the multiset $\set{\beta_1,\ldots,\beta_r}$ and $h(X) = [\Lambda : \Gamma]$, where $\Gamma$ is the lattice generated by $X$ and $\Lambda = (\Gamma \otimes \setr) \cap \setz^n$ is the maximal integral lattice of dimension $\abs{X}$ containing $\Gamma$.
\end{thm}

The subsets $X$ in \cref{thm:ehr} correspond exactly to the spanning hyperforests of $H$, since $\tilde H_d(F, \setr) = 0$ if and only if the vectors $\partial_d((v_0v_1\cdots v_d))$ for $\set{v_0,v_1,\ldots,v_d} \in E(F)$ are linearly independent. For each such $F$, we have $\Gamma = B_{d-1}(F; \setz)$. Recall that for any lattice $\Gamma \subseteq \setz^n$, we have
\[\abs{(\setz^n/\Gamma)_\tors} = [(\Gamma \otimes \setr) \cap \setz^n : \Gamma].\]
Therefore
\begin{multline*}\abs{\tilde H_{d-1}(F; \setz)_\tors} = \abs{(Z_{d-1}(F;\setz)/B_{d-1}(F;\setz))_\tors} \\ = \abs{(Z_{d-1}(H;\setz)/B_{d-1}(F;\setz))_\tors} = [B_{d-1}(F;\setr) \cap Z_{d-1}(H;\setz) : B_{d-1}(F;\setz)] = h(F),
\end{multline*}
as desired.
\end{proof}

\begin{cor}
The number of lattice points in $\cZ_H$ is the sum of $\abs{\tilde H_{d-1}(F; \setz)_\tors}$ over all spanning hyperforests $F$ of $H$.
\end{cor}
\begin{proof}
This is \cref{thm:zehr} with $t = 1$.
\end{proof}

We are now ready to deduce Theorem \ref{thm:main}. 

%\begin{cor}
%The volume of $\cZ_H$ is the sum of $\abs{H_{d-1}(T; \setz)}$ for all spanning hypertrees $T$ of $H$.
%\end{cor}
\begin{proof}[Proof of Theorem \ref{thm:main}]
The volume of $\cZ_H$ is the sum of $\abs{H_{d-1}(T; \setz)}$ for all spanning hypertrees $T$ of $H$. It is well-known (and easy to see by considering the $t\to\infty$ asymptotic), that for $\cZ \in \setr^n$, the volume of $\cZ_H$ is the $t^n$ coefficient in $L(\cZ, t)$. In this case, we have $\dim \tilde Z_{d-1}(H, \setr) = \binom{n-1}{d}$, so the hyperforests that contribute to the leading coefficient are exactly the hypertrees (see \cite[Prop.~2]{Kalai83}). Moreover, for any $d$-dimensional hypertree $T$, the group $\tilde H_{d-1}(T; \setz)$ is finite, so all of its elements are torsion. The result follows.
\end{proof}

\section{Sign Patterns of Coboundaries and the Face Lattice}
In this section we develop a combinatorial descriptions of the faces of $\cZ_H$.

Consider a simplicial complex $X$ and an integer $k \geq 0$, let $\tilde X_k$ denote the set of ordered $(k+1)$-tuples that are orderings of elements of $X_k$. Call a function $\sigma \colon \tilde X_k \to \set{-1,0,1}$ a \vocab{$k$-sign pattern} (or \vocab{sign pattern} if $k$ is understood) of $X$ if for every permutation $\pi$ of $\set{0,1,\ldots,k}$ and $(v_0,v_1,\ldots,v_k) \in\tilde X_k$, we have
\[\sigma(v_{\pi(0)},v_{\pi(1)},\ldots,v_{\pi(k)}) = \sgn \pi \cdot \sigma(v_0,v_1,\ldots,v_k).\]
Call a sign pattern \vocab{proper} if its range does not contain $0$.

Given a coboundary $\alpha \in B^k(X; \setr)$, one can define a sign pattern $\sgn \alpha$ to be the function $\sigma$ defined by
\[\sigma(v_0,v_1,\ldots,v_k) = \sgn\paren[\big]{\alpha((v_0v_1\cdots v_k))};\]
call a sign pattern $\sigma$ \vocab{valid} if there is some cocycle $\alpha$ with $\sgn\alpha = \sigma$. Finally, say a sign pattern $\sigma$ \emph{refines} a sign pattern $\tau$ (denoted $\sigma \succeq \tau$) if for all $v \in \tilde X_k$ with $\tau(v) \neq 0$, we have $\sigma(v) = \tau(v)$.

\begin{thm} \label{thm:face}
Let $H$ be a $(d+1)$-uniform hypergraph. Under refinement, the valid $d$-sign patterns of $H$ form a lattice that is isomorphic to the face lattice of $\cZ_H$ ($\succeq$ corresponds to $\subseteq$). Under this isomorphism, the dimension of a face corresponding to a valid sign pattern $\sigma$ is
\[\dim \spanop \setmid[\big]{\partial_d((v_0v_1 \cdots v_d))}{\sigma(v_0,v_1,\ldots,v_d) = 0}.\]
\end{thm}
\begin{proof}
As with the proof of \cref{thm:zehr}, we appeal to a more general result about zonotopes.
\begin{thm}[{see e.g.\ \cite[Sec.~7.3]{Ziegler}}] \label{thm:zonoface}
If $\cZ$ is the Minkowski sum of the line segments $\conv(\set{0, \beta_i})$ for $\beta_1, \beta_2, \ldots, \beta_r \in \setr^n$, then the faces of $\cZ$ are in bijection with tuples of the form 
\[(\sgn (\gen{\gamma,\beta_1}),\ldots,\sgn (\gen{\gamma,\beta_r})) \in \set{-1,0,1}^r\]
for $\gamma \in \setr^n$. Moreover,
\begin{itemize}
\item The dimension of a face corresponding to a tuple $(s_1,\ldots,s_r)$ is $\dim\spanop\setmid{\beta_i}{s_i = 0}$.
\item A face corresponding to a tuple $s \in \set{-1,0,1}^r$ contains a face corresponding to a tuple $s' \in \set{-1,0,1}^r$ if and only if $s$ is obtained from $s'$ by changing some coordinates to zeroes.
\end{itemize}
\end{thm}
Fix an arbitrary ordering of the vertices of every edge of $H$. Then, \cref{thm:zonoface} shows that the face lattice of $\cZ_H$ is isomorphic to the lattice
\[L = \setmid[\Big]{\paren[\Big]{\sgn \paren[\big]{\gen[\big]{\gamma,\partial_d((v_0v_1\cdots v_d))}}}_{(v_0,\ldots,v_d) \in E(H)}}{\gamma \in Z_{d-1}(H;\setr)^\vee} \subseteq \set{-1,0,1}^{E(H)},\]
where $V^\vee$ denotes the dual of a vector space $V$. It is equivalent to consider $\gamma \in C_{d-1}(H; \setr)^\vee = C^{d-1}(H; \setr)$. However, by definition,  
\[\gen[\big]{\gamma,\partial_d((v_0v_1\cdots v_d))} = d_{d-1} \gamma((v_0v_1 \cdots v_d)),\]
so in fact
\[L = \setmid[\Big]{\paren[\Big]{\sgn \paren[\big]{\alpha((v_0v_1\cdots v_d))}}_{(v_0,\ldots,v_d) \in E(H)}}{\alpha \in B^d(H;\setr)}.\]
A $d$-sign pattern of $H$ is uniquely determined by its values at each of its edges under our fixed ordering, so $L$ is equivalent to the set of valid sign patterns. The partial order described in the second bullet of \cref{thm:zonoface} translates to our notion of refinement, while $\dim \spanop \setmid{\beta_i}{s_i = 0}$ translates to
\[\dim \spanop \setmid[\big]{\partial((v_0v_1 \cdots v_d))}{\sigma(v_0,v_1,\ldots,v_d) = 0}.\qedhere\]
\end{proof}
\begin{cor}
The vertices of $\cZ_H$ are in bijection with the valid proper sign patterns of $B^d(H; \setr)$. Two vertices are connected by an edge if and only if their corresponding patterns differ in only one edge of $H$.
\end{cor}
\begin{proof}
Observe that the expressions $\partial_d((v_0v_1 \cdots v_d))$ for $\set{v_0,\ldots,v_d} \in E(H)$ are nonzero and no two (corresponding to distinct edges) are multiples of each other. As a result, under the bijection of \cref{thm:face}, vertices of $\cZ_H$ correspond to sign patterns with no zeroes, while edges of $\cZ_H$ correspond to sign patterns with exactly one zero (up to permutation of inputs). This immediately shows the first half of the corollary.

To show the second half, observe that if two vertices share a common edge, they are refinements of a sign pattern with one zero and must differ in only one edge of $H$. Conversely, suppose two proper valid sign patterns $\sigma_1, \sigma_2$ differ in only one edge $e$. It suffices to show that the sign pattern $\sigma_0$ with $\sigma_0(e) = 0$ but otherwise equal to $\sigma_1, \sigma_2$ must be valid, as it will correspond to an edge containing the vertices corresponding to $\sigma_1$ and $\sigma_2$. This is true as if $\sigma_1 = \sgn \alpha_1$ and $\sigma_2 = \sgn \alpha_2$, we must have $\sigma_0 = \sgn (t_1\alpha_1 + t_2\alpha_2)$ for some suitable choice of $t_1,t_2 > 0$.
\end{proof}

In the case when $H$ is the complete $(d+1)$-uniform hypergraph $K_{n}^{(d+1)}$, it turns out that the vertices of $\cZ_H=\cA_{n,d}$ are in one to one correspondence with the $d$-dimensional acyclic hypertournaments on $n$ vertices, introduced by Linial and Morganstern in \cite{LinMor}. In the language of this paper, these may be defined as follows: a \vocab{$d$-dimensional hypertournament} $T$ on $n$ vertices is defined by a set of oriented edges $E(T) \subseteq C_d(K^{(d+1)}_n; \setz)$ obtained by adding exactly one of $\set{(v_0v_1\cdots v_d), -(v_0v_1\cdots v_d)}$ for every $\set{v_0,v_1,\ldots,v_d} \subseteq [n]$. A hypertournament $T$ is called \vocab{acyclic} if the convex cone spanned by $E(T)$ does not contain a nonzero cycle, i.e.\ a nonzero element of $Z_d(K^{(d+1)}_n; \setr)$.
\begin{cor}
The vertices of $\cA_{n,d}$ are in bijection with the $d$-dimensional acyclic hypertournaments on $n$ vertices. Two vertices are connected by an edge if and only if they differ in the orientation of only one edge.
\end{cor}
\begin{proof}
Note that $d$-dimensional hypertournaments are naturally in bijection with proper sign patterns of $K_{n}^{(d+1)}$, where we associate $T$ with the unique sign pattern $\sigma$ such that
\[(v_0v_1\cdots v_d) \in T \iff \sigma(v_0, \ldots, v_d) = 1.\]
It suffices to show that acyclic hypertournaments correspond to valid sign patterns.

If $T$ corresponds to $\sgn\alpha$ for $\alpha \in B^d(K^{(d+1)}_n; \setr)$, then every element of $E(T)$ must lie in the open half-space $\setmid{\beta \in C_d(K^{(d+1)}_n; \setr)}{\alpha(\beta) > 0}$. Since $\alpha$ sends all cycles to zero, we conclude that the only cycle in the convex cone generated by $E(T)$ is $0$.

Conversely, suppose $T$ is acyclic. Then, the convex hull of $E(T)$ is compact and disjoint from the linear subspace $Z_d(K^{(d+1)}_n; \setr)$. Hence, by the hyperplane separation theorem, there must exist a linear functional $\alpha \in C^d(K^{(d+1)}_n; \setr)$ that vanishes on $Z_d(K^{(d+1)}_n; \setr)$ but is positive on every element of $E(T)$. Vanishing on $Z_d(K^{(d+1)}_n; \setr)$ is equivalent to being a coboundary,\footnote{This is a basic fact in homological algebra over a field. For a quick proof, note that $\alpha$ descends to a linear functional on $C_d(K^{(d+1)}_n;\setr)/Z_d(K^{(d+1)}_n;\setr) \cong B_{d-1}(K^{(d+1)}_n;\setr)$, so there exists some $\gamma \in B_{d-1}(K^{(d+1)}_n;\setr)^\vee$ such that $\alpha = \gamma \circ \partial_d$. We then extend $\gamma$ arbitrarily to $C_{d-1}(K^{(d+1)}_n;\setr)$ to yield a map $\gamma'$ such that $\alpha = d_{d-1}\gamma'$.} so $T$ corresponds to the valid proper sign pattern $\sgn \alpha$.
\end{proof}
\begin{rmk}
This proof is a variant of an argument in \cite[Sec.~2.1]{LinMor} showing that acyclic hypertournaments are in bijection with regions of a certain hyperplane arrangement.
\end{rmk}

\section{Concluding Remarks}
\subsection{Duality}
It is a result of Kalai \cite{Kalai83} that if $n$ and $d$ are positive integers such that $n \geq d+3$, then given any $d$-dimensional hypertree $T$ on $n$ vertices there is a dual $(n-d-2)$-dimensional hypertree $T'$ on $n$ vertices such that $\tilde H_{d-1}(T; \setz) \cong \tilde H_{n-d-3}(T'; \setz)$. It follows from \cref{thm:main} that $\cA_{n,d}$ and $\cA_{n,n-d-2}$ have the same volume.

At the same time, for any $(d+1)$-uniform $H$ on $n$ vertices, we have $\dim Z_{d-1}(H; \setr) = \binom{n-1}{d}$. Therefore $\cA_{n,d}$ and $\cA_{n,n-d-1}$ live in spaces with the same dimension. This is a curious off-by-one discrepancy, which could be interesting to explore further.

\subsection{Facets of \texorpdfstring{$\cA_{n, d}$}{A\_(n,d)}}
In light of the relation between the vertices of $\cA_{n,d}$ and acyclic hypertournaments, it it is plausible that other combinatorial aspects of $\cA_{n,d}$ have interesting properties as well. One possible object of study are the facets of $\cA_{n,d}$, which by \cref{thm:face}, correspond to the nonzero valid sign patterns with a maximal set of zeroes.

The facets of $\cA_{n,1}$ correspond to the $2^n-2$ ways to split the $n$ vertices into two nonempty sets $A$ and $B$, with the corresponding sign pattern $\sigma$ being given by $\sigma(a,b) = 1$ and $\sigma(b,a) = -1$ for $a\in A$ and $b \in B$ and zeroes everywhere else. The natural generalization to partitions of $[n]$ into $d+1$ nonempty sets yields a family of facets of $\cA_{n,d}$, but this characterization is not complete. For example, $\cA_{5,2}$ has a facet corresponding to the sign pattern $\sigma$ given by
\[\sigma(v_i, v_{i+1}, v_{i+2}) = 1\quad\text{and}\quad\sigma(v_i, v_{i+1}, v_{i+3}) = 0,\]
where indices are taken modulo $5$.

\section*{Acknowledgments}
We would like to thank Maya Sankar for helpful conversations, as well as Ernie Croot and the Fulton County Superior Court for making this collaboration possible. 

C.P.\ was supported by NSF Award DMS-2246659. D.Z.\ was supported  by the NSF Graduate Research Fellowships Program (grant number: DGE-2039656). 

\printbibliography
\end{document}